\documentclass[14pt]{extarticle} 
\usepackage[T1]{fontenc}
\usepackage[utf8x]{inputenc}
\usepackage{babel}
\usepackage[unicode=true,pdfusetitle,
bookmarks=true,bookmarksnumbered=false,bookmarksopen=false,
breaklinks=false,pdfborder={0 0 1},backref=false,colorlinks=true, allcolors=blue]
{hyperref}
\usepackage{geometry}
\usepackage{enumitem}

\input{xy}
\xyoption{all}
\usepackage[all]{xy}
\usepackage{graphics,graphicx}

\usepackage{verbatim}
\usepackage{float}
\usepackage{amsmath}
\usepackage{mathtools}
\usepackage{amssymb}
\usepackage{amsthm}

\usepackage{geometry}

\theoremstyle{plain}
\newtheorem{thmmain}{Theorem}

\newtheorem{thm}{\textsf{\textbf{Theorem}}}[section]
\newtheorem*{thm*}{\textsf{\textbf{Theorem}}}
\newtheorem{lem}[thm]{\textsf{\textbf{Lemma}}}
\newtheorem{cor}[thm]{\textsf{\textbf{Corollary}}}
\newtheorem{prop}[thm]{\emph{Proposition}}
\newtheorem*{prop*}{\emph{Proposition}}

\newtheorem*{claim*}{\emph{Claim}}

\newtheorem*{step*}{\emph{Step}}
\counterwithin*{claim}{thm}

\theoremstyle{definition}
\newtheorem{dfn}[thm]{\textbf{\textsf{Definition}}}
\newtheorem{notation}[thm]{\textbf{\textsf{Notation}}}

\newtheorem{rem}[thm]{{\textsf{Remark}}}

\newcommand{\cc}[1]{\mathcal #1}

\newcommand{\N}{\mathbb{N}}
\newcommand{\R}{\mathbb{R}}

\newcommand{\plane}{\mathbb R^{2}}

\DeclareMathOperator{\inter}{Int}

\DeclareMathOperator{\cl}{Cl}

\DeclareMathOperator{\diam}{diam}

\newcommand{\fr}{\partial}

\DeclareMathOperator{\sgnt}{sgn}
\newcommand{\sgn}[1]{\sgnt(#1)}
\DeclareMathOperator{\segg}{seg}

\newcommand{\seg}[1]{\segg(#1)}

\newcommand{\StableMan}[1]{W^{s}_{#1}}
\newcommand{\UnstableMan}[1]{W^{u}_{#1}}
\newcommand{\UnstableLocMan}[1]{W^{u}_{\text{loc},#1}}
\newcommand{\StableLocMan}[1]{W^{s}_{\text{loc},#1}}

\newcommand{\length}{l}
\newcommand{\Ryn}{\mathbb{R}^{1}_{-}}
\newcommand{\Ryp}{\mathbb{R}^{1}_{+}}
\newcommand{\Rn}{\mathbb{R}_{-}}
\newcommand{\Rp}{\mathbb{R}_{+}}
\newcommand{\Rc}{\mathbb{R}_{0}}
\newcommand{\Ryc}{\mathbb{R}^{1}_{0}}
\newcommand{\lozil}{f}
\newcommand{\lozi}{L}

\title{Strange attractors for the generalized Lozi-like family}
\author{Przemysław Kucharski}
\date{}

\providecommand{\mscclass}[1]
{
	\small	
	\textbf{\textit{MSC class---}} #1
}
\providecommand{\affiliations}[1]
{
	\small	
	\textbf{\textit{Author's affiliation}---} #1
}
\begin{document}
	\maketitle
	\date{}
	\affiliations{Faculty of Mathematics and Computer Science,
		Jagiellonian University in Krak\'{o}w
		ul. {\L}ojasiewicza 6, 30-348 Krak\'{o}w, Poland}\\
	\mscclass{37D45, 37E20}
	\begin{abstract}
		We generalize the Lozi-like family introduced in \cite{Lozi-likemaps}. The generalized Lozi-like family encompasses in particular certain Lozi-like maps \cite{Lozi-likemaps}, orientation preserving or reversing Lozi maps or large parameter regions of 2-dimensional border collision normal forms \cite{simpson:robust-chaos}. We prove that it possesses strange attractors, arising as a homoclinic class.
	\end{abstract}
	\section*{Introduction}
	Since the influential paper by Lorenz in 1963 \cite{lorenz:deterministic-flow}, hinting on what is now commonly referred to as chaos, strange attractors has been an object of intensive study. Nowadays, it is evident that systems possessing a strange attractor are important to the theory and applications (for references see for example \cite{joshi2016generalized}). As seen in the case of Lorenz \cite{lorenz:deterministic-flow}, Chua \cite{chua} or R\"ossler attractors \cite{ROSSLER1976397}, such systems are very often not amenable to rigorous mathematical analysis, even in presence of strong numerical evidence. Consequently, as a step towards better understanding Lorenz attractor, H\'enon introduced in \cite{henon1976} the H\'enon family, likely because of its similarity to a Poincaré section of a Lorenz attractor. Since then study of attractors of discrete time dynamical systems evolved on its own into rich, independent theory. Nonetheless, despite, in many cases, their simple and direct definitions, discrete time dynamical systems exhibiting strange attractors still pose a significant challenge. For example, it took enormous amount of effort to prove existence of strange attractors for H\'enon maps \cite{henon-dynamics-of}.
	
	In this context piecewise affine homeomorphisms of surfaces, showing certain degree of  hyperbolicity, were studied as a stepping stone towards more general families \cite{buzzi,strange-attractor-mis,simpson:robust-chaos,glendinning2019robust}, approximating behaviour of more complex maps, possibly sharing some features. Nonetheless, even on their own they offer objects which are approachable while still possessing complex dynamics, which study can enrich our knowledge of dynamical systems in general. Probably one of the first studied examples was the Lozi family. It was introduced by Ren\'e Lozi in 1978 \cite{lozi-first-article} as a simplification of the H\'enon family. He conjectured, based on numerical results, that it possesses a strange attractor for certain pair of parameters. In 1980 Michał Misiurewicz provided mathematical explanation of observed phenomena, presenting idea of proof of transitivity that was later used on at least a few occasions \cite{simpson:robust-chaos,simpson-gosh:robust-chaos,rychlik:invariant-measures,Lozi-likemaps,kucharski:strange-attractors}. Today, much is known about the Lozi family and much effort was put into finding its generalizations, potentially showing some differences in dynamics (see \cite{Lozi-likemaps} and discussion about kneading sequences). Broader family that is still possible to analyse using methods developed for the Lozi family should bring us closer to characterization-like theorems. In \cite{banerjee:robust-chaos} Soumitro Banerjee, James A. Yorke and Celso Grebogi discussed the 2-dimensional border-collision normal forms, a $4$-parameter piecewise affine planar maps generalizing the Lozi family, they expected that it has a strange attractor for an open set of parameters, where it is orientation preserving. It was rigorously proved in \cite{simpson:robust-chaos} by Paul Glendinning and David Simpson, but the authors used weaker definition of a strange attractor then the one present for example in works by Misiurewicz and Sonja \v{S}timac \cite{Lozi-likemaps,strange-attractor-mis}. In \cite{Lozi-likemaps} \v{S}timac and Misiurewicz introduced a family called the Lozi-like maps, modelling its properties on the Lozi family, and proved that it has strange attractors. The said family went beyond piecewise affine maps, allowing the possibility of certain kind of piecewise smooth maps. Let us note that their work was based on the paper by Misiurewicz \cite{strange-attractor-mis} and so Lozi-like maps were meant to be orientation reversing. Hence Lozi-like maps do not encompass the 2-dimensional border-collision normal forms and so their results cannot be used to strengthen findings of Glendinning and Simpson \cite{simpson:robust-chaos}. 
	
	The primary goal of this article is to find family of piecewise smooth planar homeomorphisms that would be applicable to both the Lozi-like maps and 2-dimensional border-collision normal forms. In doing so we propose conditions that are somewhat similar to ones defining Lozi-like maps and show that strange attractors are always by them implied, proving the following.
	\begin{thmmain}\label{theorem:attractor-as-homoclinic-class}
		Let $\cc{F}=\{\lozil_{\mu}\}_{\mu\in M}$ be a generalized Lozi-like family. Then we can find an open parameter region $\cc{U}\subset M$ such that for any $\mu\in \cc{U}$, the map $\lozil:=\lozil_{\mu}$ has a chaotic attractor $\cc{A}:=\bigcap_{n\in\N}\lozil^{n}(F)$, on which $\lozil|_{\cc{A}}$ is mixing. Moreover, if $\cc{F}$ is orientation reversing, $\cc{A}$ is maximal, that is it is a strange attractor. If $\cc{F}$ is orientation reversing, has a renormalization model with a $\gamma$-tangency curve at $\mu_{0}\in M_{0}\cap\cc{U}$, then $\cc{A}$ is maximal in some open parameter region in $\cc{U}$ with $\mu_{0}$ in the boundary.
	\end{thmmain}	
	Our exposition follows reasoning presented in the author's previous work \cite{kucharski:strange-attractors} and uses renormalization model from \cite{dyi-shing-ou:critical-points-I}. As the majority of ideas, needed for the proof of orientation preserving case of Theorem \ref{theorem:attractor-as-homoclinic-class}, were introduced in \cite{kucharski:strange-attractors}, the purpose of current work is introduction of proper definitions and showcasing that renormalization, the main tool used in the proof of \cite[Thm. B]{kucharski:strange-attractors}, is persistent under even not so small perturbations, given we stay in certain subfamily of the Lozi-like family. Moreover, we do not present any applications of Theorem \ref{theorem:attractor-as-homoclinic-class} and leave it for future works. It is clear that showing the renormalization can be applied to border collision normal forms and orientation reversing Lozi family is a matter of rudimentary, although possibly tedious, calculations.

	\section{Definitions}
	\subsection{Notions of attractor}
	We will strengthen results of Glendinning, Simpson and Gosh \cite{simpson-gosh:renormalization,simpson:robust-chaos}. We follow the definition introduced by John Milnor \cite{milnor:on-the-concept-of-attractor}. Let $h\colon X\to X$ be a continuous map on a topological space $X$, and $F\subset X$ be a closed subset. We define the realm of attraction $\rho(F)$ of $F$ as all those points $x\in X$ such that $\omega(x)\subset F$. 
	\begin{dfn}\label{def:attractor}
		A subset $A\subset X$ will be called an attractor if it satisfies the following\begin{enumerate}[label=(A\arabic*)]
			\item its realm of attraction $\rho(A)$ is of positive Lebesgue measure,
			\item for any closed $A'\subset A$ with $\rho(A')=\rho(A)$ we must have $A'=A$
		\end{enumerate}
	\end{dfn}
	Some authors require from the attractor additional properties. Nonetheless, we will be working with the above definition, as it allows wide variety of subsets to be attractors and so usually does not discriminate other definitions in use. Main object in this paper is an attractor that exhibits some chaotic properties. In papers on border collision normal form \cite{simpson:robust-chaos} and Lozi-like maps \cite{Lozi-likemaps} relevant to results in this article, there are two definitions of attractor with chaotic properties used. Let us first specify a few notions. An open subset $U\subset X$ will be called a trapping region if $h(\cl U)\subset U$. A compact set $F\subset X$ contained in a trapping region $U$ and satisfying $\bigcap_{n\in\N}h^{n}(U)=F$ will be called maximal in $U$. A map $h$ is said to be transitive if for any open sets $W,V\subset X$ we can find $i\in\N$ such that $h^{i}(W)\cap V\neq\emptyset$. On the other hand, if $h^{i}(W)\cap V\neq\emptyset$ holds for almost every $i\in\N$, then $h$ is said to be mixing. We are ready to state.
	\begin{dfn}\label{def:strange-chaotic-att}
		A compact, invariant subset $A\subset X$ will be called a strange attractor if it has a trapping region in which it is maximal. On the other hand, if $A\subset$ is an attractor such that $h|_{A}$ is transitive and periodic points are dense in $A$, $A$ will be called a chaotic attractor and said to have Devaney chaos.
	\end{dfn}
	\subsection{Universal invariant cones}
We begin by recalling necessary definitions from \cite{Lozi-likemaps}. For a unit vector $v\in\plane$ and a number $l\in(0,1)$ we define a cone $K$ in $\plane$ by \begin{equation}\label{eq:cone-dfn}
K:=\{u\in\plane\colon|\langle u,v \rangle|\geq l||u||\},
\end{equation} where $||\cdot||$ denotes the Euclidean norm and $\langle\cdot,\cdot\rangle$ a scalar product. The straight line $\R v$ is the axis of the cone. Two cones will be called disjoint if their intersection consists only of the vector $0$. Let $U\subset \plane$ be an open subset, we define a cone-field $\cc{C}$ on $U$ as the assignment of a cone $K_{z}$ to each point $P\in U$ such that the axis $v_{z}$ and the coefficient $l_{z}$ vary continuously with $z\in U$. In this way $z\mapsto v_{z}$ forms a vector field of axes' directions. 
\begin{dfn}
	We call a pair of cones $K^{u}$ and $K^{s}$ universal if they are disjoint and the axis of $K^{u}$ is the $x$-axis, and the axis of $K^{s}$ is the $y$-axis.
\end{dfn}
\begin{notation}
	Throughout the rest of the paper we will denote by $\alpha_{u}$ and $\alpha_{s}$ the unstable and stable coefficients of universal cones. Their axes will be denoted by $v_{u}=(1,0)$ and $v_{s}=(0,1)$.
\end{notation}
\begin{dfn}
	A family of cones $\cc{C}=\{K_{z}\}_{z\in\plane}$, where \[K_{z}=\{u\in\plane\colon|\langle u,v_{z} \rangle|\geq l_{z}||u||\}\] will be called a cone family if $z\mapsto l_{z}$ and $z\mapsto v_{z}$ are measurable transformations. The field $\cc{C}$ will be called a stable family of cones, respectively an unstable family of cones, if $K_{z}\subset K^{s}$, respectively $K_{z}\subset K^{u}$, for every $z\in\plane$.
\end{dfn}
\begin{dfn}
	Let $\gamma_{u}$, $\gamma_{s}$ be piecewise $C^{1}$ curves with tangent sections contained in respectively stable and unstable cone families. Then $\gamma_{u}$ will be called an $u$-curve and $\gamma_{s}$ an $s$-curve.
\end{dfn}
The following lemma can be found in \cite{Lozi-likemaps}. As it is crucial for the definition of the Lozi-like family, we present it below.
\begin{lem}\label{lemma:cone->gamma-a-function}
	For any $u$-curve $\gamma_{u}$ one can find a $C^{1}$ function $\phi_{u}\colon I_{u}\to \R$, where $I_{u}$ is an interval or $\R$, such that $\gamma_{u}=\{(x,\phi_{u}(x))\colon x\in I_{u}\}$. Similarly, any $s$-curve can be expressed as $\gamma_{s}=\{(\phi_{s}(y),y)\colon y\in I_{s}\}$ for some $C^{1}$ function $\phi_{s}\colon I_{s}\to \R$. Moreover, $\phi_{u}$, $\phi_{s}$ are Lipschitz continuous with constants, respectively $\sqrt{1-l^{2}}/l$ and $\sqrt{1-r^{2}}/r$, where $l$, $r$ are the coefficients in definitions of cones, respectively $K^{u}$ and $K^{s}$.
\end{lem}
We will say that a curve as above is infinite if the domain of the corresponding function is all of $\R$, if the domain is an interval we will say that the curve is compact.

We will also need the following.
\begin{lem}\label{lemma:uniqueness-of-int-Ku-Ks}
	Let $K^{u}$ and $K^{s}$ be a universal pair of cones. Let $\gamma_{u}$ and $\gamma_{s}$ be smooth infinite $u$-curve and $s$-curve. Then $\gamma_{u}$ and $\gamma_{s}$ intersect at exactly one point.
\end{lem}
\begin{dfn}\label{def:directions}
	Let $L$ and $R$ be subsets of $\R$. Then $L$ is said to be to the left of $R$, denoted $L\lhd_{u} R$, if for any $u$-curve $\gamma_{u}=\{(x,\phi_{u}(x))\}_{x\in\R}$ the intersection $L\cap \gamma_{u}$ is less then $R\cap \gamma_{u}$ in the ordering on $\gamma_{u}$ induced from the projection $(x,\phi_{u}(x))\mapsto x$. Similarly, we will say that a sets $B$ is below the $U$, and write $B\lhd_{u} U$, if for any $s$-curve $\gamma_{s}=\{(\phi_{s}(y),y)\}_{y\in\R}$ the intersection $U\cap\gamma_{s}$ is less then $U\cap\gamma_{s}$ in the ordering on $\gamma_{s}$ induced from $(\phi_{s}(y),y)\mapsto y$. Ordering $\rhd_{s}$ and $\rhd_{u}$ are defined as reverse to $\lhd_{s}$ and $\rhd_{u}$. We treat here sets as a free family over $\{1,-1\}$, and set the rule $-1 A \lhd_{s} 1 B$ if and only if $-1 A \rhd_{s} 1 B$ for $A,B\subset \plane$. Similar convention will be used for other orderings.
\end{dfn}
\begin{rem}
	The ordering on $s$ or $u$-curve $\gamma$ induced from the projections as in Definition \ref{def:directions} will be referred to as natural ordering and denoted by $\lhd_{\gamma}$.
\end{rem}
\begin{dfn}\label{def:rectanle}
	Let $h\colon J:=I_{u}\times I_{s}\to \plane$ be a homeomorphism. Then $h(J)$ will be called a rectangle. Faces $h(\{0\}\times I_{s})$, $h(\{1\}\times I_{s})$ will be called vertical and faces $h(I_{u}\times\{0\})$, $h(I_{u}\times\{1\})$ horizontal. If vertical faces are $s$-curves and horizontal ones are $u$-curves, $h(J)$ will be called a $(u,s)$-rectangle. Moreover, vertical faces will be called left and right, according to their relative position and Def. \ref{def:directions}. Similarly, horizontal faces will be called upper and lower. A $s$-strip or $u$-strip that has corresponding interval comprising of a point will be called a $s$-curve or $u$-curve on $R$. If $I_{u}=[0,\infty)$, the rectangle will be said to be infinite to the right.
\end{dfn}
\begin{dfn}\label{def:strip}
	Let $h\colon [0,1]^{2}\to \plane$ define an $(u,s)$-rectangle $R=h(\{0\}\times[0,1])$. Let $\gamma_{s,1}$ and $\gamma_{s,2}$ be finite $s$-curves with endpoints connecting upper and lower faces of $R$, then a subset \[S(\gamma_{s,1},\gamma_{s,2}):=\{z\in R\colon \gamma_{s,1}\lhd_{u}z \lhd_{u} \gamma_{s,2}\},\] will be called an $s$-strip on $R$. Similarly, for $u$-curves $\gamma_{u,1}$, $\gamma_{u,2}$, we define a $u$-strip \[U(\gamma_{u,1},\gamma_{u,2}):=\{z\in R\colon \gamma_{u,1}\lhd_{s}z \lhd_{s} \gamma_{u,2}\}.\] A $s$-strip of a rectangle will be called proper it is a proper subset of $R$ and does not intersect  vertical faces of $R$. Analogously, a $u$-strip will be proper if it is a proper subset of $R$ that does not intersect horizontal faces of $R$.
\end{dfn}

\subsection{The generalized Lozi-like family}\label{section:lozi-like-dfn}
\begin{dfn}
	Cone fields $\cc{C}^{u}_{\mu}=\{K^{u}_{z,\mu}\}_{z\in\plane,\mu\in M}$ and $\cc{C}^{s}_{\mu}=\{K^{s}_{z,\mu}\}_{z\in\plane,\mu\in M}$ will be said to vary continuously if $(z,\mu)\mapsto (v_{u,z,\mu},v_{s,z,\mu})$ and $(z,\mu)\mapsto (l_{u,z,\mu},l_{s,z,\mu})$ are continuous.
\end{dfn}
\begin{dfn}
	Let $\lozil\colon \plane\to\plane$ be a $C^{1}$ diffeomorphism. Then $\lozil$ will be said to have invariant cone fields $\cc{C}^{u}=\{K^{u}_{z}\}_{z\in\plane}$ and $\cc{C}^{s}=\{K^{s}_{z}\}_{z\in\plane}$ subject to universal cones $K^{u}$ and $K^{s}$ with expansion $\lambda\in\R$ if the following holds
	\begin{enumerate}[label=(C\arabic*)]
		\item Cone field have their associated vector fields of axes $v_{u}=\{v_{u,z}\}_{z\in\plane}$ and $v_{s}=\{v_{s,z}\}_{z\in\plane}$, and coefficients $\{l_{u,z}\}_{z\in\plane}$, $\{l_{s,z}\}_{z\in\plane}$, that vary continuously with $z\in\plane$.
		\item For every point $z\in\plane$ we have $K^{u}_{z}\subset K^{u}$, ${df}_{z}(K^{u}_{z})\subset K^{u}_{f(z)}$ and $K^{s}_{z}\subset K^{s}$, ${df}^{-1}_{z}(K^{s}_{z})\subset K^{s}_{f^{-1}(z)}$,
		\item For every point $z\in\plane$, we have $||df(\textbf{w})||\geq \lambda ||\textbf{w}||$ for every $\textbf{w}\in K^{u}_{i}$ and $||df^{-1}(\textbf{w})||\geq \lambda ||\textbf{w}||$ for every $\textbf{w}\in K^{s}$.
	\end{enumerate}
	Moreover, if $\cc{F}=\{\lozil_{\mu}\}_{\mu\in M}$ is a family of $C^{1}$ diffeomorphisms satisfying the above, we will say that the family has cone fields $\cc{C}^{u}_{\mu}$ and $\cc{C}^{s}_{\mu}$ subject to universal cones, if $\cc{C}^{u}_{\mu}$ and $\cc{C}^{s}_{\mu}$ are continuous.
\end{dfn}
\begin{dfn}
	Let $f_{1},f_{2}\colon \plane\to\plane$ be $C^{1}$ diffeomorphisms. We say that $f_{1}$ and $f_{2}$ are $\lambda$-synchronous if
	\begin{enumerate}[label=(S\arabic*)]
		\item $f_{1}$ and $f_{1}$ are either both orientation reversing, or both orientation preserving 
		\item both $f_{1}$ and $f_{2}$ have the same invariant cone fields subject to universal cones with expansion constant $\lambda$.
		\item There exists a smooth curve $\Rc\subset\plane$, called the divider, which is an infinite $s$-curve, and its image is an infinite $u$-curve.
	\end{enumerate}
\end{dfn}
The divider splits the plane into two components $\Rn$ and $\Rp$ which we call, respectively, the left half-plane and the right half-plane. Similarly, $\Ryc:=f_{i}(\Rc)$ splits the plane into two components called the lower half-plane and the upper half-plane. We will denote them by $\Ryn$ and $\Ryp$. We may assume that directions suggested in names of half planes corresponds to the notion of directions introduced by Definition \ref{def:directions} and that $\lozil(\Rn)=\Ryn$, and $\lozil(\Rp)=\Ryp$. Let us also recall the notion of a stable and unstable manifold. Let $f\colon U\to \plane$ be a continuous bijection on an open set $U\subset \plane$. We define the local stable and unstable manifolds at $z\in\plane$ as, respectively
\[\StableLocMan{x,\epsilon}:=\{q\in\plane\colon d(f^{n}(z),f^{n}(q))\to0\text{ as }n\to\infty\text{, and  }d(z,q)<\epsilon\}\]\[
\UnstableLocMan{x,\epsilon}:=\{q\in\plane\colon d(f^{-n}(z),f^{-n}(q))\to0\text{ as }n\to\infty\text{, and  }d(z,q)<\epsilon\}.\] By $\StableLocMan{z}$ and $\UnstableLocMan{z}$ we will denote the local stable and unstable manifolds with $\epsilon>0$ maximal so that they are of class $C^{1}$. The global invariant manifolds are denoted by $\UnstableMan{z}$ and $\StableMan{z}$. 
	\begin{dfn}\label{definition:lozi-like-maps}
		Let $\lozil_{-}, \lozil_{+}\colon\plane\to\plane$ be $\lambda$-synchronous  $C^{1}$ diffeomorphisms with the divider $\Rc$. Let $\lozil\colon \plane\to\plane$ be defined by the formula\[
		\lozil|_{\cl\Rn}:=\lozil_{-}|_{\cl\Rn}\text{ and  }\lozil|_{\cl\Rp}:=\lozil_{+}|_{\cl\Rp}.
		\]We call $\lozil$ a gluing of $\lozil_{-}, \lozil_{+}$. Moreover, we call it a Lozi-like map if the following hold
		\begin{enumerate}[label=(L\arabic*)]
			\item\label{L1} $|\det d\lozil_{i,z}|<1$ for every $z\in\plane$ and $i=1,2$
			\item\label{L2} there exists a trapping region $G$ that has closure homeomorphic to a closed disc and $\inter G\cap \Rc\neq\emptyset$
			\item\label{L3} the expansion constant is big enough $\lambda>\sqrt{2}$.
			\item\label{L4} there is a fixed point $X\in G$ such that $\lozil(G)$ is disconnected by $\StableLocMan{X}$ into three components: one in the upper half plane, one in the lower half plane and one in the right half plane, with $\lozil(G\cap \Rc)\subset \Rp$.
		\end{enumerate}
	\end{dfn}
	A family, continuously dependent on $\mu$, $\{\lozil_{\mu}\}_{\mu\in M}$ of Lozi-like maps will be called a Lozi-like family. The above definition is the same as in \cite{Lozi-likemaps} and can be considered quite general, encompassing also family introduced by Young in \cite{young:SRB-for-piecewise-hyperbolic-maps}. Usually one needs to impose additional restrictions in order to derive properties present in the Lozi family. For use it will mean requiring certain geometric properties that later will allow us to introduce a renormalization model as in \cite{dyi-shing-ou:critical-points-I}. Effectively, it will prove that the renormalization model of Ou is robust under perturbations, as long as we stay in certain subfamily of Lozi-like maps. Let us introduce mentioned conditions and later take care of the existence of the model.
	\begin{dfn}\label{definition:ll-with-renormalization}
		A Lozi-like map $\lozi\colon\plane\to\plane$ will be called a Lozi-like map with a renormalization model, if the following conditions hold
		\begin{enumerate}[label=(R\arabic*)]
			\item\label{R1} there exists an invariant infinite to the right $(u,s)$-rectangle $\vec{R}$ based on a homeomorphism $r\colon[0,1]\times [0,\infty)\to\plane$ 
			\item\label{R2} the rectangle $R:=r([0,1]^{2})\supset G$ is an $(u,s)$-rectangle with left, right, lower and upper faces of $R$, satisfying respectively $R_{l}\subset \Rn$, $R_{r}\subset \Rp$, $R_{-}\subset \Ryn$, $R_{+}\subset \Ryn$, and with images positioned in such a way that $\lozil(R_{l})\subset R_{l}$ and $\lozil(\vec{R})\cap \fr R\cap \Rp \subset \Ryp\cap R_{l}$,
			\item\label{R3} there exists a fixed point $X\in\Rp$ such that its stable local manifold $\StableLocMan{X}\cap R\subset\Rp$ intersects upper and lower faces of $R$
			\item\label{R4} the set $\lozil(\vec{R})\cap \vec{R}$ is disconnected by $\StableLocMan{X}$ into three components: one $(u,s)$ rectangle in the upper half plane, another one in the lower half plane and a closed disc in the right half plane
		\end{enumerate}
	\end{dfn}
	A family $\{\lozil_{\mu}\}_{\mu\in M}$, that depends continuously on $\mu$ in $M$, and satisfying the above conditions, will be called a Lozi-like family with a renormalization model. Note that the left face $R_{l}$ is invariant and so necessarily contains a fixed point, denoted here by $Y$ and $Y\in $. It has a local unstable manifold on which it must be orientation preserving. Moreover, as tangent vectors to $R_{l}$ are contained in the stable cone family, $R_{l}\subset \StableLocMan{Y}$. Depending on whether the Lozi-like map preserves or reverses orientation, $\lozil(R_{-})$ may be to the right or the left of $\lozil(R_{-})$. It is not hard to convince yourself that if $\lozil$ preserves the orientation then $\lozil(R_{-})$ is to the right of $\lozil(R_{+})$. In the orientation reversing case, $\lozil(R_{-})$ is to the left of $\lozil(R_{+})$. To accommodate both the cases we can simply write $\epsilon\lozil(R_{+})\lhd_{u}\lozil(R_{-})$, where $\epsilon=-1$ if $\lozil$ reverses the orientation and $\epsilon=1$ otherwise. The variable $\epsilon$ will be fixed throughout the rest of the paper in the above manner. For later use, let us explicitly state one of conclusions from the discussion above.
	\begin{rem}\label{remark:orientation-on-unst}
		Note that regardless of whether a Lozi-like map preserves or reverses orientation in general, it always reverses the orientation on $\UnstableMan{X}$. Indeed, otherwise one could not hope to find an invariant region containing $X$, as one branch of $\UnstableMan{X}$ would escape to infinity.
	\end{rem}

	\begin{lem}\label{lem:orientation-on-stable-y}
		A Lozi-like map is orientation preserving if and only if it preserves orientation on $\StableLocMan{Y}$ if and only if it reverses orientation on $\StableLocMan{X}$. Similar statement holds for the orientation reversing case.
	\end{lem}

	\begin{lem}\label{lem:flipping-of-lozilike}
		Let $R$ be an $(u,s)$-rectangle. Denote by $\gamma_{s,1}$ and $\gamma_{s,2}$ its left and right faces, and by $\gamma_{u,1}$, $\gamma_{u,2}$ its lower and upper faces, that is $\gamma_{u,1}\lhd_{s}\gamma_{u,2}$. Suppose one can find an $s$-curve $\alpha$ such that $\lozil(\gamma_{s,1})\cup \lozil(\gamma_{s,2})\subset \alpha$ and $\lozil(\gamma_{s,1})\lhd_{\alpha}\lozil(\gamma_{s,2})$, then $-\epsilon\lozil(\gamma_{u,1})\lhd_{s}\lozil(\gamma_{u,2})$.
	\end{lem}
	We will need one more notion, that is not present in Misiurewicz and \v{S}timac work, but can be found in Young's paper \cite{young:SRB-for-piecewise-hyperbolic-maps} and essentially accommodates the idea of a Lozi-like map being close to one dimension, allowing for certain perturbation arguments. Our definition can be seen as a bridge between analytic exposition of Young and more geometric one of Misiurewicz and \v{S}timac.
	\begin{dfn}\label{definition:lozi-like-family}
		A family $\cc{F}=\{\lozil_{\mu}\}_{\mu\in M}$ of gluing of synchronously hyperbolic maps will be called a Lozi-like family if $\cc{F}$, its invariant cone families and divider depend continuously on the parameter, and one can find subsets $M_{1}$, $M_{0}\subset M$, $M_{1}$ open, such that $\cc{F}|_{M_{1}}$ comprises of Lozi-like maps. Moreover, $\cc{C}^{u}|_{M_{1}}$ and $\cc{C}^{s}|_{M_{1}}$ have non-zero cones coefficients, and $\cc{C}^{u}|_{M_{0}}$ and $\cc{C}^{s}|_{M_{0}}$ have identically zero cones coefficients. That is, invariant cone families of $\cc{F}|_{M_{0}}$ degenerate to fields of axis. Additionally, $\cc{F}$ will be called a Lozi-like family with a renormalization if $\cc{F}|_{M_{1}}$ are Lozi-like maps with a renormalization.
	\end{dfn}
	\begin{rem}
		Note that a Lozi-like family will usually fail to contain Lozi-like maps in $M_{0}$. It is also the case for the Lozi family, which satisfies Definition \ref{definition:lozi-like-maps} only in a certain parameter region. The presence of ambient space of parameters is necessary if we want to use perturbation arguments.
	\end{rem}
	\section{The renormalization model}
	Assume $\lozil\colon\plane\to\plane$ is a Lozi-like map with a renormalization model. We closely follow the exposition by Ou \cite{dyi-shing-ou:critical-points-I}. Let $R$ be the invariant rectangle. We first define pull-backs of $s$-curves on $R$. Note that $\lozil(R)$ is folded along $\lozil(\Rc)$ and $\lozil(R\cap\Rc)$ is a $u$-curve. Let $u_{L}$ be the left utmost point of $\lozil(R\cap\Rc)$ and $u_{R}$ its right utmost point. Moreover, let $\omega$ be an $s$-curve on $R$. If $\omega\cap\lozil(R\cap\Rc) $ is to the left of $u_{L}$, then $\lozil^{-1}(\omega)\cap R$ contains two $s$-curves, one is $\Pi_{-}(\omega)$ and the other is $\Pi_{+}(\omega)$, where $\Pi_{\sigma}(\omega):=\lozil_{\sigma}^{-1}(\omega)\cap R$. In this way, the transformations $\Pi_{-}$ and $\Pi_{+}$ define pull-backs of $s$-curves by the two branches of $\lozil$. We define $s$-strips on $R$. Recall that $\StableLocMan{X}\cap R$ is an $s$-curve on $R$, connecting lower and upper faces of $R$. Let $\beta_{1}=\StableLocMan{X}\cap R$, and $\beta_{m}=\Pi_{-}(\beta_{m-1})$ for $2\leq m< \infty$, $\beta_{\infty}=R_{l}$, $\gamma_{m}=\Pi_{+}(\beta_{m})$ for $1\leq m <\infty$ and $\gamma_{\infty}=\Pi_{+}(\beta_{\infty})=R_{r}$. The $s$-curves $\{\beta_{m}\}_{1\leq m<\infty}$ are subsets of $\StableMan{X}$, while, as mentioned before, $\beta_{\infty}\cup\gamma_{\infty}\subset \StableMan{Y}$. Let $B=S(\beta_{2},\beta_{1})$, $C=S(\gamma_{1},\gamma_{\infty})$, $C_{m}=S(\gamma_{m-1},\gamma_{m})$ for $2\leq m < \infty$, and $D=S(\beta_{\infty},\gamma_{\infty})$. The sets $\{C_{m}\}_{2\leq m <\infty}$ form a partition of $C$.
	
	\begin{prop}\label{prop:renormalization-main}
		The first return map $h\colon C\to C$ is well defined and the following holds for $n\geq 2$.
		\begin{enumerate}[label=(\roman*)]
			\item\label{renorm:prop1} the sets $C_{n}$, are non-empty $(u,s)$-rectangles with boundary comprising of curves $\fr C_{n}^{u,-}$, $\fr C_{n}^{u,+}$ of $\UnstableMan{Y}$, positioned respectively in the lower and upper half plane, and curves $\fr C_{n}^{s,l}$, $\fr C_{n}^{s,r}$ of $\StableMan{X}$, with $\fr C_{n}^{s,l}$ to the left of $\fr C_{n}^{s,r}$,
			\item\label{renorm:prop3} $C_{n}$ is to the left of $C_{n+1}$, and $C_{n'}\cap C_{m'}=\fr C_{n'}^{s,r}=\fr C_{m'}^{s,l}$ for $n'=m'-1$ and otherwise empty
			\item\label{renorm:prop4} $h_{n}:=h|_{C_{n}}=f^{n}|_{C_{n}}$ and $h\colon C_{n}\to U_{n}$ is piecewise $C^{1}$ on two pieces,
			\item\label{renorm:prop5} there is an $s$-curve $S_{n}\subset \inter C_{n}$ with endpoints on $\fr C_{n}^{u,-}$ and $\fr C_{n}^{u,+}$ for which $h_{n}$ is $C^{1}$ on components $C_{n}^{l}$ and $C_{n}^{r}$ of $C_{n}\setminus S_{n}$. Moreover, $h(S_{n}):=T_{n}\subset \Ryc$,
			\item\label{renorm:prop6} $C_{n}^{l}$ is to the left of $C_{n}^{r}$,
			\item\label{renorm:prop7} $C_{n}$ is a union of two rectangles bordering on $S_{n}$
			\item\label{renorm:prop8} $U_{n}$ is a union of two  rectangles, namely $U_{n}^{+}:=h(C_{n}^{l})$ and $U_{n}^{-}:=h(C_{n}^{r})$ bordering on $T_{n}$, positioned, respectively, in the  upper and lower half planes,
			\item\label{renorm:prop9} the order $\lhd_{u}$ on $\{U_{n}\}_{n\geq 2}$ is isomorphic via $U_{n}\mapsto -(\lambda^{-1}\epsilon)^{n-1}$ with the order on $\{-(\lambda^{-1}\epsilon)^{n-1}\}_{n\geq 2}$ induced from the real line,
			\item\label{renorm:prop10} $-\epsilon^{n-1}\fr U_{n}^{u,r}\lhd_{u} \fr U_{n}^{u,l}$, where $\fr U_{n}^{u,l}:=h(\fr C_{n}^{u,-})$ and $\fr U_{n}^{u,r}:=h(\fr C_{n}^{u,+})$.
		\end{enumerate}
	\end{prop}
	\begin{proof}
		Claims from \ref{renorm:prop1} to \ref{renorm:prop3} are direct consequences of the construction. To show \ref{renorm:prop4} we simply note that $S_{n}$ are consecutive pull-backs of $\Rc$ via $\Pi_{-}$ and $\Pi_{+}$, similarly as $\beta_{n}$ and $\gamma_{n}$. In a direct way points \ref{renorm:prop5} and \ref{renorm:prop6} follow from the above. The point \ref{renorm:prop7} follows from analogous property for the rectangle $R$, given by condition \ref{R4}. The hardest part of the proof is to show the last two claims. With that in mind, let us first note that $\UnstableLocMan{Y}$ is contained in the lower half plane and intersects $\Ryc$ in the right half plane. It is a consequence of \ref{R4} and the fact that $\StableLocMan{X}$ is contained in the right half plane. Therefore, every $\beta_{m}$, $m=1,2,...,$ intersects $\UnstableLocMan{Y}$ transversally in the lower half plane. Define inductively $K_{2}:=\lozil(C)$ and then $K_{m}:=\lozil(K_{m-1})\setminus f^{m}(C_{m})$ for $m>2$. It follows from the construction that $K_{m}$ is a $(u,s)$-rectangle with the left face on $\StableLocMan{Y}$ and the right face contained in $\beta_{1}$. Hence, every $\beta_{i}$, $i=1,...$, intersects both the upper and lower faces of every $K_{m}$, $m=2,...$, transversally. We define the signed distance between $K_{m}$ and $\UnstableLocMan{Y}$ along $\beta_{i}$ by $\theta_{i,m}:=\tau_{s}\length(s)$, where $s$ is the curve contained in $\beta_{i}$ connecting $K_{m}$ and $\UnstableLocMan{Y}$ with the smallest length, and $\tau_{s}=1$ if the corresponding curve is above $\UnstableLocMan{Y}$ and $\tau_{s}=-1$ otherwise. Now, it follows from the existence of stable invariant cone family that \begin{equation}\label{eq:2}
		|\theta_{i,m}|\leq\lambda^{-1}|\theta_{i+1,m-1}|
		\end{equation} for $i\geq 1$. Moreover, as $\lozil$ preserves or reverses orientation on $\beta_{i}$ according to the sign of $\epsilon$ we also have \begin{equation}\label{eq:1}
		\sgn{\theta_{i,m}}=\epsilon\sgn{\theta_{i+1,m-1}}.
		\end{equation} Iterating \eqref{eq:1} we obtain $\sgn{\theta_{1,m}}=\epsilon^{m-2}\sgn{\theta_{m-1,2}}=\epsilon^{m-2}$. As a consequence of \eqref{eq:2}, $K_{m}$ is closer to $\UnstableLocMan{Y}$ than $K_{m'}$ if $m>m'$. Such $K_{m}$ and $K_{m'}$ must be disjoint and connect $\beta_{\infty}$ with $\beta_{1}$. It follows that the order $\lhd_{s}$ on $\{f^{m-1}(C_{m})\}_{m\geq 2}$ is isomorphic with the natural one on $\{\lambda^{-m}\epsilon^{m-2}\}_{m\geq 2}$ via $f^{m-1}(C_{m})\mapsto \lambda^{-m}\epsilon^{m-2}$. It now remains to apply Lemma \ref{lem:flipping-of-lozilike} to conclude that the order on $\{U_{m}\}_{m\geq 2}$ is isomorphic with the natural order on $\{-(\lambda^{-1}\epsilon)^{m-1}\}_{m\geq 2}$, proving Claim \ref{renorm:prop8}. The last claim follows similarly from Lemma \ref{lem:flipping-of-lozilike}. Let us write down a few details. The upper face of $K_{2}$ is $\lozil(\fr C^{+})$, while the lower is $\lozil(\fr C^{-})$, where $\fr C^{+}$ and $\fr C^{-}$ are upper and lower faces of $C$. Upper and lower faces of $K_{m}$ are, therefore, $\lozil^{m-1}(\fr C^{+})$ and $\lozil^{m-1}(\fr C^{-})$. We have $\lozil(\fr C^{-})\lhd_{s}\lozil(\fr C^{+})$. As, after $m-2$ iterates under $\lozil$, the order of $\lozil(\fr C^{-})$ and $\lozil(\fr C^{+})$ is reversed or preserved, depending on $\epsilon$, $m-2$ times, we have $\lozil^{m-1}(\fr C^{-})\epsilon^{m-2}\lhd_{s}\lozil^{m-1}(\fr C^{+})$. Applying Lemma \ref{lem:flipping-of-lozilike} to $\lozil^{m-1}(\fr C^{u,-}_{m})=\lozil^{m-1}(\fr C^{-})\cap B$ and $\lozil^{m-1}(\fr C^{u,+}_{m})=\lozil^{m-1}(\fr C^{+})\cap B$, we obtain $-\epsilon^{m-1}\lozil^{m}(\fr C^{u,-}_{m})\lhd_{u}\lozil^{m}(\fr C^{u,+}_{m})$, proving the claim.
	\end{proof}
	We are now ready to explain what is meant by a $\gamma$-tangency curve.
	\begin{dfn}\label{definition:gamma-tangency-curve}
		Let $\cc{F}$ be a Lozi-like family with a renormalization and $u_{\mu}$ be the first turn of $\UnstableMan{Y}$ inside $R$ for the parameter $\mu\in M$, that is $\{u_{\mu}\}=\lozil(\UnstableLocMan{Y}\cap \Rc)$. Then $\lozil$ is said to have a $\gamma$-tangency curve at $\mu_{0}\in M_{0}$, if one can find a parameter curve $\kappa\colon [0,1)\to M_{1}$ such that $u_{\kappa(t)}\in \gamma_{i}$ for some $i\geq 2$ and every $t\in[0,1)$, and $\kappa(t)=\mu_{0}$.
	\end{dfn}
	\section{Density of the stable manifold of $X$}\label{section:density-of-homoclinic-intersections}
		In this section we show that the stable manifold of $X$ is dense in $R$. The proof essentially does not differ from the one presented in \cite{Lozi-likemaps}, although we present more details of necessary computations.
		
		For a piecewise smooth curve $\gamma\colon I:=[0,1]\to \plane$ let $\gamma'=(\gamma_{x}',\gamma_{y}')$ be the section of vectors tangent to $\gamma$ restricted to its domain $\bar{I}\subset I$ of differentiability. Formally, $d\gamma(\dfrac{d}{dt})=\gamma'$, where $d\gamma$ is a differential of $\gamma$ defined whenever $\gamma$ is smooth, and $\dfrac{d}{dt}$ is a vector generating the tangent space $TI$ of $I$. We define the length of $\gamma$ by $\length(\gamma):=\int ||\gamma'||$. Moreover, whenever we write $\gamma'\subset \cc{C}$ and $\cc{C}=\{K_{z}\}$ is a cone family, we will mean that $\gamma'(t)\in K_{\gamma(t)}$ for every $t\in\bar{I}$. We begin the proof by the following lemma.
	\begin{lem}\label{lemma: smooth-arcs-bounded}
		Any smooth arc $\gamma\subset R$ with $\gamma'\subset \cc{C}^{u}$ satisfies $\length(\gamma)\leq \alpha_{u}^{-1} \diam R$.
	\end{lem}
	\begin{proof}
		Let $\gamma=(\gamma_{x},\gamma_{y})\colon[0,1]\to\plane$ satisfy the assumptions. Recall, we denoted by $\alpha_{u}$ the coefficient of $K^{u}$ and $v_{u}:=(1,0)$ the axis of $K^{u}$. From the definition, \begin{equation}\label{eq:no-smooth-1}
		||\gamma'||\leq \alpha_{u}^{-1} |\langle \gamma',v_{u}\rangle|.
		\end{equation} Note that, as $\gamma'_{t}\subset K_{\gamma(t)}^{u}\subset K^{u}$, by Lemma \ref{lemma:cone->gamma-a-function} $\gamma$ is a graph of a function  with respect to the $x$-axis. That is \[ \phi\colon \gamma(\bar{I})_{x}\ni\gamma_{x}(t)\mapsto \gamma_{y}(t)\in \gamma(\bar{I})_{y} \] is a function, where $(\cdot)_{x}$ and $(\cdot)_{y}$ are projections on $x$-axis and $y$-axis. Moreover, $\gamma_{x}(t)=\gamma_{x}(t_{0})+\int_{s=t_{0}}^{t}\langle \gamma'(s),v_{u}\rangle ds$, where $t\in(t_{0},t_{1})\subset\bar{I}$, and $t_{0}$, $t_{1}$ are endpoints of a maximal interval in $\bar{I}$, and so the length of $\gamma(\bar{I})_{x}$ is equal to $|\int\langle \gamma',v_{u}\rangle|$. As $\langle \gamma',v_{u}\rangle$ is either positive or negative on $(t_{0},t_{1})$, the length of $\gamma(\bar{I})_{x}$ is equal to $\int|\langle \gamma',v_{u}\rangle|$. Integrating both sides of \eqref{eq:no-smooth-1}, since $\gamma(\bar{I})_{x}\subset R_{x}$, we obtain \[
		\length(\gamma)\leq \alpha_{u}^{-1}|\gamma(\bar{I})_{x}|\leq \alpha_{u}^{-1} \diam R,\]which proves the claim.
	\end{proof}
	\begin{lem}\label{proposition:st-mani-X-intersects-arcs}
		If the condition \ref{R4} is satisfied, the stable manifold of $X$ intersects any smooth arc $\gamma\subset R$ with $\gamma'\subset \cc{C}^{u}$.
	\end{lem}
	\begin{proof}Let $\gamma\subset R$ satisfy $\gamma'\subset\cc{C}^{u}$. By condition \ref{R4} it is enough to prove that $\lozil^{i}(\gamma)$ intersects both $\Rc$ and $\Ryc$. After $2i$ iterations of $\gamma$ we $\lozil^{2i}(\gamma)$ will comprise of at most $2^{i}$ smooth curves, each with tangent vectors contained in $\cc{C}^{u}$. Thus, by Lemma \ref{lemma: smooth-arcs-bounded}, $\length(\lozil^{2i}(\gamma))\leq 2^{i}\alpha_{u}^{-1}\diam R$. On the other hand, since $\gamma'\subset\cc{C}^{u}$, integral computation shows $\length(f^{2i}(\gamma))\geq \lambda^{2i}\length(\gamma)$, which gives us \[
	\lambda^{2i}\length(\gamma)\leq	\length(f^{2i}(\gamma))\leq2^{i}\alpha_{u}^{-1}\diam R.
	\]Consequently, \[
	\lambda\leq	\sqrt{2}\left[(\length(\gamma)\alpha_{u})^{-1}\diam R\right]^{1/2i},
	\]which contradicts the condition \ref{L3}.
	\end{proof}
	Following same line of reasoning as in \cite{Lozi-likemaps} or \cite{strange-attractor-mis} we prove the following.
	\begin{cor}\label{corollary:stmani-dense-in-R}
		The stable manifold $\StableMan{X}$ is dense in $R$. Moreover, $\StableMan{X}$ intersects every arc of unstable manifold in $R$ and $\lozil|_{\cl\UnstableMan{X}}$ is mixing.
	\end{cor}
	\section{Proof of Theorem \ref{theorem:attractor-as-homoclinic-class}}
	\subsection{The orientation preserving case}\label{section:existence-main-results}	
	The most non-trivial case to deal with is the orientation preserving case. In this section we assume that $\cc{F}=\{\lozil_{\mu}\}_{\mu\in M}$ is an orientation preserving Lozi-like family with a renormalization and a $\gamma$-tangency curve at $\mu_{0}\in M$. We will usually drop the dependence on $\mu$ and write simply $\lozil$. Let $A$ be the first turn of $\UnstableLocMan{Y}$ inside $R$. From the condition \ref{R4} it follows that $\UnstableLocMan{Y}$ intersects $\StableLocMan{X}$ at $D$ in the lower half plane. As $X$ lies in the upper half plane, and $\lozil$ reverses the orientation of $\StableLocMan{X}$, $E:=\lozil(D)\in \Ryp$ and points $E$, $D$, and $A$ forms a triangle $H_{0}$. To apply the renormalization model in the proof of Theorem \ref{theorem:attractor-as-homoclinic-class}, we consider first return map $\hat{h}\colon H_{0}\to H_{0}$, developing ideas in \cite{kucharski:strange-attractors}. Let $\hat{C}_{i}:=C_{i}\cap H_{0}$, $\hat{U}_{i}:=U_{i}\cap H_{0}$. One can find the least natural numbers $p=p(\mu)\in\N$ and $q=q(\mu)\in\N$ such that $\hat{U}_{p}\neq\emptyset$ and $\hat{U}_{q}\cap \cc{A}\neq\emptyset$, of course $q\leq p$. Put $H=\bigcup _{i\in\N}\lozil^{i}(H_{0})$. Then $H$ is a finite union, comprising of all Kakutani-Rokhlin towers for $h$. Note that if $i< p$, the set $\hat{C}_{i}$ is a $(u,s)$-rectangle, while the set $\hat{C}_{p}$ is a triangle. Let us introduce the notation for boundaries as follows. The left, right, upper and lower faces of $\hat{C}_{i}$ will be denoted, respectively, by $\fr\hat{C}^{l}_{i}$, $\fr\hat{C}^{r}_{i}$, $\fr\hat{C}^{+}_{i}$ and $\fr\hat{C}^{-}_{i}$. Moreover, it follows directly from the construction that $\hat{U}_{i}$ and $\hat{C}_{i}$, for $i<p$, inherit from the family $\{U_{i}\}_{i\geq 2}$ analogous properties as those in Proposition \ref{prop:renormalization-main} with $\epsilon=1$. Following convention that sets intersected with $H_{0}$ are written with hat, we can restate Proposition \ref{prop:renormalization-main} in the case of the first return map $\hat{h}\colon H_{0}\to H_{0}$. While we leave to the reader to find full statement of the proposition, let us write down the most important part. Namely, the order $\lhd_{u}$ is linear on $\{\hat{U}_{i}\}_{2\leq i \leq p}$, that is $\hat{U}_{i}\lhd_{u}\hat{U}_{i+1}$ for $2\leq i <p$. It remains to determine the behaviour of $\hat{h}_{p}$ in the presence of $\gamma$-tangency curve. We split discussion into two lemmas. The first one follows directly from the construction of renormalization model.
	\begin{lem}\label{lemma:perturbing-lozi-like}
		Let $\{\mu_{i}\}_{i\in\N}$ be a sequence of parameters with $\lim_{i\to\infty}\mu_{i}\in M_{0}$. Then $\diam\lozil(h\cap \Ryc)\to 0$ as $i\to\infty$.
	\end{lem}
	\begin{proof}
		Let us consider parameters $\mu\in M_{0}$. Note that fields of axis define integral curves, which in turn give rise to the unstable foliation of $\plane$. Let $L_{Y}$ be the leaf going through fixed point $Y$. Then $L_{Y}$ must be invariant and, since $L_{Y}=\StableMan{Y}$, $Y\in\Ryc$. Therefore, the unstable manifold $\UnstableMan{Y}\subset\Ryc$. It implies that $\UnstableMan{Y}$ is close to $\Ryc$ in the Hausdorff metric for $\mu$ close to $M_{0}$. Consequently, $E$ and $D$ are close to $X$, and so $H_{0}$ is contained in some small neighbourhood of $\Ryc$. As $H$ is a finite union of iterates of $H_{0}$, $\diam\lozil(h\cap \Ryc)\to 0$ as $i\to\infty$.
	\end{proof}
	A direct consequence of the above lemma and Definition \ref{definition:gamma-tangency-curve} is the following.
	\begin{lem}\label{lem:prelim-tangency-param}
		One can find an open set $\cc{U}_{\text{tan}}$ with $\mu_{0}\in \fr \cc{U}_{\text{tan}}$ such that if $\mu\in \cc{U}_{\text{tan}}$ then either
		\begin{enumerate}[label=(T\arabic*)]
			\item\label{tangency-c1} $q(\mu)=p(\mu)=n$ and $\hat{h}(H_{0})\cap\Ryc\subset\hat{C}^{r}_{p}$, or
			\item\label{tangency-c2} $q(\mu)=p(\mu)=n+1$ and $\hat{h}(H_{0})\cap\Ryc\subset \hat{C}^{l}_{p}$.
		\end{enumerate}
		Moreover, we have $\hat{h}(H_{0}\cap\Ryc)\subset \hat{C}_{2}$.
	\end{lem}
	\begin{lem}\label{lem:behaviour-hp}
		Let $\cc{F}=\{\lozil_{\mu}\}_{\mu\in M}$ be a Lozi-like family with renormalization and a $\gamma$-tangency curve at $\mu_{0}\in M$. Then for parameters in $\cc{U}_{\text{tan}}$ either $\hat{U}_{p}\subset \hat{C}_{2}$ or the following holds
		\begin{enumerate}[label=(P\arabic*)]
			\item\label{Cp->Up(1)} there is a curve $\hat{R}_{p}\subset \inter \hat{C}_{p}$ with endpoints on $\fr \hat{C}_{p}^{u,-}$ and $\fr \hat{C}_{p}^{u,+}$ for which $h_{p}$ is affine on components $\hat{C}_{p}^{l}$ and $\hat{C}_{p}^{r}$ of $\hat{C}_{p}\setminus R_{p}$. Moreover, $\hat{h}(\hat{R}_{p}):=\hat{T}_{p}\subset \{y=0\}$,
			\item\label{Cp->Up(2)} $\hat{U}_{p}$ is a broken triangle, positioned to the right of $\hat{U}_{n}$, for $n=2,...,p-1$, and comprising of a rectangle $\hat{U}_{p}^{-}:=\hat{h}(\hat{C}_{p}^{r})$ and a triangle $\hat{U}_{p}^{+}:=\hat{h}(\hat{C}_{p}^{l})$, positioned respectively in the lower and upper half plane
			\item\label{Cp->Up(3)}  the boundary $\fr \hat{U}_{p}$ comprises of the curve $\hat{h}(\fr \hat{C}^{s,l}_{p})\subset \StableMan{X}$ and broken curves $\fr \hat{U}_{p}^{u,l}=\hat{h}(\fr \hat{C}_{p}^{u,-})$ and $\fr \hat{U}_{p}^{u,r}=\hat{h}(\fr \hat{C}_{p}^{u,+})$ with the former lying to the left of the latter,
			\item\label{Cp->Up(4)} the triangle $\hat{U}_{p}^{+}$ connects $\hat{C}_{p}$ with $\hat{C}_{2}$ and $\hat{h}(A)\in \hat{C}_{2}$.
		\end{enumerate}
	\end{lem}
	\begin{proof}
		Let us first note that \ref{Cp->Up(1)} follows simply from the renormalization model. Moreover, if $\mu\in \cc{U}_{\text{tan}}$ and \ref{tangency-c1} occurs, then $\hat{U}_{p}\subset \hat{C}_{2}$ follows from $\hat{h}(H_{0}\cap\Ryc)\subset \hat{C}_{2}$. Let us now assume that \ref{tangency-c2} occurs. Note that \ref{Cp->Up(2)} is always satisfied whenever \ref{tangency-c2} and the moreover part of Lemma \ref{lem:prelim-tangency-param} holds, it is a consequence of Proposition \ref{prop:renormalization-main} and the fact that $\lozil$ preserves orientation. On the other hand, to prove \ref{Cp->Up(3)} we simply consider the boundary $\fr \hat{C}_{p}$ as an oriented curve and use the fact that $\lozil$ preserves its orientation. Hence, positions of curves of $\lozil(\fr\hat{C}_{p})=\fr\hat{U}_{p}$ must be as stated in \ref{Cp->Up(3)}. At last, \ref{Cp->Up(4)} follows directly from \ref{tangency-c2} and the moreover part of Lemma \ref{lem:prelim-tangency-param}.
	\end{proof}

	\begin{prop}
		For parameters in $\cc{U}_{\text{tan}}$ we have $\bigcap_{i\in\N}\hat{h}^{i}(H_{0})=H_{0}\cap\cl \UnstableMan{X}$.
	\end{prop}
	\begin{proof}
		Note that to prove the proposition it is enough to show $\cc{R} (\bar{i},\bar{t}):=\bigcap_{k\in\N}\hat{h}^{i_{k}}(\hat{C}_{t_{k}})\subset \cl \UnstableMan{X}$, where $\bar{i}=\{i_{k}\}_{k\in\N}$ is any increasing sequence with $i_{0}=0$ and $\bar{t}=\{t_{k}\}_{k\in\N}$ satisfies $t_{k}\in\{1,..,p\}$. For later use, we define $\cc{R} (\bar{i},\bar{t},n):=\bigcap_{k\leq n}\hat{h}^{i_{k}}(\hat{C}_{t_{k}})$ for $n\in\N$.
		
		Moreover, it is visible that if $\hat{h}^{j}(\hat{C}_{t'})\subset W$ for some $W$ with $\omega(W)\subset \cl\UnstableMan{X}$, $j\in\N$ and $t'\in\{1,...p\}$, and if $t'=t_{k}$ for infinitely many $k\in\N$, we must have $\cc{R}(\bar{i},\bar{t})\subset \cl\UnstableMan{X}$. At last, note that any closed disc $W$ with $\fr W\subset \UnstableMan{X}\cup\StableMan{X}$, must also have $\omega(W)\subset\cl\UnstableMan{X}$.
		
		Note that by Lemma \ref{lem:prelim-tangency-param} one can find a curve $\gamma\subset \hat{C}_{p-1}\cap \UnstableMan{X}$ which connects $\fr \hat{C}^{u,l}_{p-1}$ and $\fr \hat{C}^{u,r}_{p-1}$. Thus, $\hat{h}(\gamma)$ connects $\hat{h}(\fr \hat{C}^{u,l}_{p-1})$ and $\hat{h}(\fr \hat{C}^{u,r}_{p-1})$, and by \ref{renorm:prop6} $\hat{h}(\gamma)$ is to the right of every $\hat{U}_{j}$, for $j=2,...,p-1$. Hence, $\hat{h}(R)\subset W$, where $W$ is bounded by $\seg{ED}\subset\StableMan{X}$ and $\gamma\subset \UnstableMan{X}$, and $R=H_{0}\setminus (\hat{C}_{p-1}\cup \hat{C}_{p})$. It proves $\cc{R}(\bar{i},\bar{t})\subset\cl\UnstableMan{X}$ in the case $t_{k}<p-1$ for $k\in\N$. 
		
		We split our discussion into two cases, according to conditions of Lemma \ref{lem:behaviour-hp}. Assume first, that $\hat{C}_{p}\subset \hat{U}_{2}$. Then, by our previous discussion, since $\hat{U}_{2}\subset R$, there must be $\cc{R}(\bar{i},\bar{t})\subset \cl\UnstableMan{X}$ if $t_{k}=p$ for infinitely many $k\in\N$. Consider now sets $\cc{R}(\bar{i},\bar{t})$ with $t_{k}= p-1$ for infinitely many $k\in\N$. It is enough to treat sets with $t_{k}= p-1$ for every $k\in\N$. We will show that any sequence $\bar{i}$ defines $\cc{R}(\bar{i},\bar{t},n)$, $n\geq 0$, which is a disjoint union of rectangles $T$ and every $T$ has its opposite sides connected by $\UnstableMan{X}$. As $\hat{h}$ is area contracting, it shows $\cc{R}(\bar{i},\bar{t})\subset\cl\UnstableMan{X}$. Using the last claim of Lemma \ref{lem:prelim-tangency-param} we see that $\cc{R}(\bar{i},\bar{t},1)$ is a disjoint union of rectangles $T$ with mentioned form. Note that the image $\hat{h}(T)$ of a rectangle $T$ is a union of two rectangles bordering on $\hat{T}_{p-1}$, both connecting $\hat{h}(\fr \hat{C}_{p-1}^{s,l})$ with $\hat{h}(\fr \hat{C}_{p-1}^{s,r})$. Moreover, $\hat{h}(\hat{T})\cap  \hat{T}_{p-1}\subset \hat{C}_{p}$. Consequently, $\hat{h}(\hat{T})\cap \hat{C}_{p-1}$ is a union of two rectangles with prescribed properties. Proceeding by induction, we see that one can apply the above reasoning for any $n\in\N$, proving our claim.
		
		We turn to the case, when properties \ref{Cp->Up(1)} to \ref{Cp->Up(4)} hold. Using Lemma \ref{lem:prelim-tangency-param} and property \ref{Cp->Up(4)} one can find $\tau\subset \hat{C}_{p}\cap\UnstableMan{X}$ with $\hat{h}(\tau)$ connecting $\seg{ED}$, $\hat{T}_{p}$ and $\fr \hat{C}_{2}^{s,r}$. Let $g\in \hat{h}(\tau)\cap \fr \hat{C}_{p-1}^{s,r}$. Let $S_{g}\subset \fr \hat{C}^{s,r}_{p-1}$ be the segment connecting $g$ and $\fr \hat{C}^{u,-}_{p-1}$. Consider the region $W'$ bounded by $\seg{ED}$, $\hat{h}(\tau)$, $S_{g}$ and $\bigcup_{j=2}^{p-1}\fr \hat{C}^{u,-}_{j}$. Note that $\omega(\fr W')\subset \cl\UnstableMan{X}$, since $\omega(R)\subset\cl\UnstableMan{X}$. It follows from \ref{Cp->Up(2)} and \ref{renorm:prop9} that $\hat{h}(\hat{C}_{j})\subset W'$ for every $j=2,...,p-1$. Hence, $\cc{R}(\bar{i},\bar{t})\subset\cl\UnstableMan{X}$ if only $t_{k}<p$ for infinitely many $k\in\N$. Therefore, it remains to consider the case of $t_{k}=p$ for every $k\in\N$. As before, it is not hard to prove, using induction, Lemma \ref{lem:behaviour-hp} and \ref{lem:prelim-tangency-param} that, for every $n\in\N$, $\cc{R}(\bar{i},\bar{t}, n)$ is a union of rectangles $T$ which have opposite sides connected by a curve in $\UnstableMan{X}$. Consequently, as $h$ contracts the area, $\cc{R}(\bar{i},\bar{t})\subset\cl\UnstableMan{X}$, proving the proposition.
	\end{proof}
	
	We have an immediate
	\begin{cor}\label{trapping-region}
		For parameters in $\cc{U}_{\text{tan}}$ we have $\bigcap_{n=0}^{\infty}f^{n}(H)=\cl\UnstableMan{X}$.
	\end{cor}
	As a consequence we obtain
	
	\begin{thm}
		For parameters in $\cc{U}_{\text{tan}}$ the unstable manifold $\cc{A}=\cl \UnstableMan{X}$ of $X$ is a topological attractor.
	\end{thm}
	\begin{proof}
		We will find an open set $V\supset \cc{A}$ with $f(\cl V)\subset V$ and $V\subset H$. It will conclude our proof, since by Theorem \ref{trapping-region} $\bigcap_{n=0}^{\infty}f^{n}(H)=\cc{A}$.
		
		As $\cc{A}\subset \inter H$, we have $f^{p+1}(H)\subset \inter H$ for some $p\in\N$. The least such number is in fact the same as in the proof of Lemma \ref{trapping-region}. Hence we can find an open set $V_{0}\subset \inter H$ with $f^{p+1}(H)\subset V_{0}$. We inductively define, for $0<n\leq p$, open sets $V_{n}$ with $\cl V_{n}\subset f^{n}(\inter H)$ and $f(\cl V_{n-1})\subset V_{n}$. Assume we have constructed $V_{n-1}$. Note that $f(\cl V_{n-1})\subset f^{n}(\inter H)$. Therefore, it is enough to take sufficiently small neighbourhood $V_{n}\subset f^{n}(\inter H)$ of $f(\cl V_{n-1})$ so that $\cl V_{n}\subset f^{n}(\inter H)$. Let $V=\bigcup_{n=0}^{p}V_{n}$. We compute \[
		f(\cl V)\subset \bigcup_{n=0}^{p}f(\cl V_{n})\subset \bigcup_{n=1}^{p}V_{n}\cup f^{p+1}(H)\subset V.
		\] Notice that $V$ may not be an open disc. Therefore to construct attracting neighbourhood of $\cc{A}$ we have to add to $V$ all bounded components of its complement.
	\end{proof}
	\subsection{The orientation reversing case}\label{section:existence-main-results-OR}
	Our definition of orientation reversing Lozi-like map essentially does not differ from the one presented in \cite{Lozi-likemaps}. As the orientation reversing case of Theorem	\ref{theorem:attractor-as-homoclinic-class} was treated in aforementioned article, we do not repeat the argument here and refer the reader to Misiurewicz and \v{S}timac's work.\\

	\textbf{Acknowledgements.}	This work was supported by the National Science Centre, Poland (NCN), ~grant no. 2019/34/E/ST1/00237. I would like to thank Jan P. Boroński, Magda Foryś-Krawiec and Sonja \v{S}timac for many fruitful discussions.
	\bibliography{lozi_map_bib}
	\bibliographystyle{alphaurl}
	
\end{document}